\documentclass[12pt]{article}
\usepackage{amssymb,amsmath,amsthm,amsfonts,enumerate, bbm}

\newtheorem{thm}{Theorem}
\newtheorem{defi}[thm]{Definition}
\newtheorem{prop}[thm]{Proposition}

\newtheorem{rmk}[thm]{Remark}
\newtheorem{lem}[thm]{Lemma}
\newtheorem{q}[thm]{Question}

\theoremstyle{example}

\newcommand{\mH}{\mathfrak{H}}
\newcommand{\CC}{\mathcal{C}}
\newcommand{\cast}{\circledast}
\newcommand{\tast}{\stackrel{t}{\ast}}
\newcommand{\tcast}{\stackrel{t}{\circledast}}

\def\Q{\mathbb {Q}}
\def\R{\mathbb {R}}

\def\C{\mathbb {C}}
\def\h{\hbar}
\def\Ht{\mH_{t}}

\def\Hht{\mH_{\h, t}}

\def\c{\cdot}
\def\cs{\cdots}
\def\ls{\ldots} 
\def\f{\frac} 
\def\d{\displaystyle}
 
\def\p{{\mathbbm p}} 
\def\z{\zeta}
\def\s{\zeta^{\star}}
\def\q{\zeta_{q}}
\def\qs{\zeta^{\star}_{q}}

\def\tq{\zeta^{t}_{q}}

\def\cp{\circ_{+}}

\def\ap{\ast_{+}}

\def\vp{\varphi}
\def\mz{\mathfrak{z}}

\title{Interpolation of $q$-analogue of multiple zeta and zeta-star values}
\date{}
\author{Noriko Wakabayashi\thanks{College of Science and Engineering, Ritsumeikan University, 1-1-1, Nojihigashi, Kusatsu-city, Shiga 525-8577, Japan\newline
E-mail address : noriko-w@fc.ritsumei.ac.jp}}

\begin{document} 
\maketitle
\begin{abstract}
We know at least two ways to generalize multiple zeta(-star) values, or MZ(S)Vs for short, which are $q$-analogue and $t$-interpolation. 
The $q$-analogue of MZ(S)Vs, or $q$MZ(S)Vs for short, was introduced by Bradley, Okuda and Takeyama, Zhao, etc. 
On the other hand, the polynomials interpolating MZVs and MZSVs using a parameter $t$ were introduced by Yamamoto. 
We call these $t$-MZVs. 

In this paper, 
we consider such two generalizations simultaneously, that is, 
we compose polynomials, called $t$-$q$MZVs, interpolating $q$MZVs and $q$MZSVs using a parameter $t$ which are reduced to $q$MZVs as $t=0$, to $q$MZSVs as $t=1$,  and to $t$-MZVs as $q \to 1$.
Then we prove Kawashima type relation, cyclic sum formula and Hoffman type relation for $t$-$q$MZVs. 

\end{abstract}
\section{Introduction}
For a formal parameter $q$ and an index $(k_1, k_2, \ls, k_l)$ of positive integers with $k_1 \geq 2$, $q$-analogues of multiple zeta and zeta-star values ($q$MZVs and $q$MZSVs, respectively, for short) are defined by 
\begin{align*}
\q(k_1, k_2, \ls, k_l) & = \sum_{m_1 > m_2 > \cs > m_l \geq 1} \f{q^{(k_1-1)m_1 + (k_2-1)m_2 + \cs +(k_l-1)m_l}}{{[m_1]}^{k_1}{[m_2]}^{k_2} \cs {[m_l]}^{k_l}} \ (\in \Q[[q]]),\\
\qs(k_1, k_2, \ls, k_l) & = \sum_{m_1 \geq m_2 \geq \cs \geq m_l \geq 1} \f{q^{(k_1-1)m_1 + (k_2-1)m_2 + \cs +(k_l-1)m_l}}{{[m_1]}^{k_1}{[m_2]}^{k_2} \cs {[m_l]}^{k_l}} \ (\in \Q[[q]]), 
\end{align*}
where $[n]$ denotes the $q$-integer $\d [n] =\f{1-q^n}{1-q}$. 
We often call $k_1+k_2+\cdots+k_l$ (resp. $l$) the weight (resp. the depth) of the index $(k_1, k_2, \ldots, k_l)$ or of corresponding zeta values. 
In the case of $l=1$, $q$MZVs and $q$MZSVs coincide and are reduced to 
$$\q(k)=\sum_{m \geq 1}\f{q^{(k-1)m}}{[m]^k}. $$ 
If $q \in \C$, $q$MZVs and $q$MZSVs are absolutely convergent in $|q|<1$. 
Taking the limit as $q \to 1$, $q$MZ(S)Vs turn into ordinary MZ(S)Vs given by
\begin{align*}
\z(k_1, k_2, \ls, k_l) & = \sum_{m_1 > m_2 > \cs > m_l \geq 1} \f{1}{{m_1}^{k_1}{m_2}^{k_2} \cs {m_l}^{k_l}} \ (\in \R),\\
\s(k_1, k_2, \ls, k_l) & = \sum_{m_1 \geq m_2 \geq \cs \geq m_l \geq 1} \f{1}{{m_1}^{k_1}{m_2}^{k_2} \cs {m_l}^{k_l}} \ (\in \R). 
\end{align*} 
The $q$MZ(S)Vs were investigated for example in Bradley\cite{B}, Okuda and Takeyama\cite{OT}, Zhao\cite{Z}. 

On the other hand, in \cite{Y}, Yamamoto introduced the interpolation polynomial of MZVs ($t$-MZVs for short) given by
\begin{align*}
\z^t(k_1, k_2, \ls, k_l) = \sum_{{\p}} \z({\p}) t^{l - \rm{dep}({\p})} \ (\in \R[t]), 
\end{align*}
where ${\rm dep}(\p)$ is the depth of index $\p$ and $\sum_{\p}$ stands for the sum where $\p$ runs over all indices of the form
$\p=(k_1\ \square\ \cs\ \square\ k_l)$ in which each $\square$ is filled by two candidates: the comma $``,"$ or the plus $``+"$.  

We consider such two generalizations of MZ(S)Vs simultaneously. 

\begin{defi}[Interpolated $q$-analogue of multiple zeta values ($t$-$q$MZVs)]
For positive integers $k_1, k_2, \ls, k_l$ with $k_1 \geq 2$, parameters $t$ and $q$, we define $t$-$q$MZVs by 
\begin{align*}
\tq(k_1, k_2, \ls, k_l) = {\sum_{\p}}' (1-q)^{k-\rm{wt}(\p)} \z_q({\p}) t^{l - \rm{dep}(\p)} \ (\in \Q[[q]][t]),  
\end{align*}
where $k=k_1+k_2+\cdots+k_l$, ${\rm wt}(\p)$ is the weight of index $\p$ and $\sum_{\p}'$ stands for the sum where 
$\p$ runs over all indices of the form  
$\p=(k_1\ \square\ \cs\ \square\ k_l)$ in which each $\square$ is filled by three candidates:
$``,"$, $``+"$ or $``-1+"$ (minus $1$ plus). 
\end{defi}
If $q \in \C$, $t$-$q$MZVs are absolutely convergent in $|q|<1$. 
Taking the limit as $q \to 1$, $t$-$q$MZVs turn into $t$-MZVs. 
We notice that $\q^0=\q$, $\q^1=\qs$ and $\tq(k) = \q(k) \ (k \geq 2)$. 

In $\S \ref{tqK}$, 
we show the following Kawashima type relation for $t$-$q$MZVs under the appropriate algebraic setup (see $\S \ref{subK}$ for details). 

\begin{thm}\label{tqKawa}
For any $m \geq 1$ and any $v, w \in \Hht y$, we have  
\begin{align*}
\sum_{\begin{subarray}{c} i+j=m \\ i, j\geq 1 \end{subarray}}Z_q^t\bigl( \varphi_\h^t(v) \tcast_\h (-tx + y - \hbar t)^{i-1}y \bigr)
Z_q^t \bigl( \varphi_\h^t(w)\tcast_\h (-tx + y - \hbar t)^{j-1}y \bigr) &\\
=-Z_q^t \bigl( \varphi_\h^t (S_\h^t)^{-1} (S_\h^t(v) \ast S_\h^t(w))\tcast_\h (-tx + y - \hbar t)^{m-1}y \bigr). 
\end{align*}
\end{thm}
\noindent
Taking the limit as $q \to 1$, 
this formula is reduced to Kawashima type relation for $t$-MZVs proved in \cite{TW}.
If $t=0$ (resp. $t=1$), this formula is reduced to Kawashima type relation for $q$MZVs (resp. $q$MZSVs) proved in \cite{T}. 

As an application of Theorem \ref{tqKawa}, the following identity called cyclic sum formula for $t$-$q$MZVs is proved in $\S \ref{tqC}$. 

\begin{thm}\label{tqCSF}
Let $k_1, k_2, \ls, k_l$ be positive integers with $(k_1, k_2, \ls, k_l) \neq (1, 1, \ls, 1)$ and put $k=k_1 + k_2 + \cdots + k_l$. Then we have  
\begin{align*}
&\sum_{i=1}^l \sum_{j=0}^{k_i-2} \tq(k_i-j, k_{i+1}, \ls, k_l, k_1, \ls, k_{i-1}, j+1)\\
& = (1-t) \sum_{i=1}^{l}\tq(k_i+1, k_{i+1}, \ls, k_l, k_1, \ls, k_{i-1}) \\
& \quad + t^l \sum_{i=0}^{l} (k-i) (1-q)^i \binom{l}{i} \tq(k-i+1). 
\end{align*}
\end{thm}
\noindent
Taking the limit as $q \to 1$, 
this formula is reduced to cyclic sum formula for $t$-MZVs proved in \cite{TW, Y}.
If $t=0$ (resp. $t=1$), this formula is reduced to cyclic sum formula for $q$MZVs (resp. $q$MZSVs) proved in \cite{B} (resp. \cite{OO}). 

As another easier application of Theorem \ref{tqKawa}, we state Hoffman type relation for $t$-$q$MZVs in $\S \ref{tqH}$. 
\section{Kawashima type relation for $t$-$q$MZVs}\label{tqK}
\subsection{Algebraic setup}\label{subK}
Let $\h$ be a formal variable.  Denote by $\Hht = \Q[\h, t] \langle x,y \rangle$ the non-commutative polynomial algebra over $\Q[\h, t]$ in two indeterminates $x$ and $y$, 
and by $\Hht^{1}$ and $\Hht^{0}$ its subalgebras $\Q[\h, t]+ \Hht y$ and $\Q[\h, t]+ x\Hht y$, respectively. 
Put $z_j = x^{j-1}y\ (j \geq 1)$.  We define the weight and the depth of a word $u=z_{k_1} z_{k_2} \cdots z_{k_l}$ by ${\rm wt}(u)=k_1+k_2+\cdots+k_l$ and ${\rm dep}(u)=l$, respectively. 

\begin{itemize}
\item 
Define the $\Q[\h, t]$-linear map $\widehat{Z}_q^t : \Hht^{0} \longrightarrow \Q[\h, t][[q]]$ by $\widehat{Z}_q^t(1) = 1$ and 
$$\widehat{Z}_q^t(z_{k_1} z_{k_2} \cs z_{k_l}) = \tq(k_1, k_2, \ls, k_l) \ \ \ (k_1 \geq 2). $$
We also define the substitution map  
$f : \Q[\h, t][[q]] \longrightarrow \Q[t][[q]]$ by $f : \h \longmapsto 1-q$ and set 
$$Z_q^t = f \circ \widehat{Z}_q^t. $$  
\end{itemize}

\begin{itemize}
\item 
Let $\mz$ be the $\Q[\h, t]$-submodule of $\mH_{\h, t}^1$ generated by
$A : =\{ z_j | j \geq 1\}$. 
We give the product $\cp$ : $\mz \times \mz \longrightarrow \mz$ characterized by $\Q[\h, t]$-bilinearity and
$$z_i \cp z_j  = z_{i+j} + \hbar z_{i+j-1} \ \ \ ( i,\ j \geq 1 ). $$ 
The product $\cp$ determines $\mz$-module structure on $\mH_{\h, t}^1$ by
$$z_i \cp 1 =0, \ z_i \cp (z_jw) = (z_i \cp z_j)w\ \ \ (w \in \Hht^1). $$
\end{itemize}

\begin{itemize}
\item 
The $\Q[\h, t]$-linear map 
$S^t_\h : \Hht^1 \longrightarrow \Hht^1$
is defined by $S_\h^t(1) = 1$ and
$$S_\h^t(aw) = aS_\h^t(w) + ta \cp S_\h^t(w), $$
where $a \in A$ and $w (\in \Hht^1$) is a word. 
\end{itemize}

We prove the following lemma immediately by the definition of $\cp$ and $S_\h^t$. 
\begin{lem}\label{Sgam}
For $a \in \mz$ and $w \in \Hht^1$, we have
\begin{itemize}
\item[(i)] $S_\h^t(a \cp w) = a \cp S_\h^t(w), $
\item[(ii)] $S_\h^t(wy)=\gamma_\h^t(w)y, $ 
\end{itemize}
where $\gamma_\h^t$ denotes the automorphism on $\Hht$ characterized by 
$$\gamma_\h^t(x) = x,\ \gamma_\h^t(y) = tx + y + \h t.$$  
\end{lem}

%
By the definition of $S_\h^t$ or Lemma \ref{Sgam} (ii), we find that  
\begin{align}
Z^t_q =  Z_q^0 S_\h^t, \label{Ztq}
\end{align}
where $Z_q^0(w) : = Z_q^t  (w) |_{t=0}$ ($w \in \Hht^0$). 

\begin{itemize}
\item 
We let 
\begin{align}
\vp_\h^t = -(S_\h^t)^{-1} \vp S_\h^t, \label{vp}
\end{align}
where $\vp$ denotes the automorphism on $\Hht$ characterized by $\vp(x) =x+y$, $\vp(y)=-y$. 

\item
The harmonic product $\ap$ on $\mH_{\h, t}^1$ is defined by the $\Q[\h, t]$-bilinearity and  
\begin{align*}
1 \ap w & = w \ap 1 = w,\\
z_i u \ap z_j v & = z_i (u \ap z_j v) + z_j (z_i u \ap v) + (z_i \cp z_j)(u \ap v) 
\end{align*}
for $i, j \geq 1$ and words $u, v, w \in \mH_{\h, t}^1$. 
\end{itemize}   
\begin{lem}\label{Sap}
\begin{itemize}
\item[(i)] For any $k, l \geq 1$, $v \in \Hht y$, $w \in \Hht^1$, we have
$$\gamma_\h^t (z_k) v \ap z_l w = \gamma_\h^t(z_k) (v \ap z_l w) +z_l \bigl( \gamma_\h^t(z_k)v \ap w \bigr) +(1-t)(z_k \cp z_l)(v \ap w). $$
\item[(ii)] For any $k, l \geq 1$, $v, w \in \Hht^1$, we have 
$$S_\h^t(z_k v) \ap z_l w = \gamma_\h^t(z_k) \bigl( S_\h^t(v) \ap z_l w \bigr) +z_l \bigl( S_\h^t(z_k v) \ap w \bigr) + (1-t) (z_k \cp z_l) \bigl( S_\h^t(v) \ap w \bigr). $$
\item[(iii)] For any $k \geq 0$, $l \geq 1$, $v=z_pV \ (p \geq 1, V \in \Hht^1)$ and $w \in \Hht y$, 
\begin{align*}
x^k v \ap \gamma_\h^t(z_l) w = z_{k+p} \bigl( V \ap \gamma_\h^t(z_{l})w \bigr) + \gamma_\h^t(z_l)(x^k v \ap w) +(1-t)(z_{k+p} \cp z_l)(V \ap w). 
\end{align*}
\end{itemize}
\end{lem}
\begin{proof}
We show (i) first. 
It is sufficient to show the case of $v=z_pV\ (p \geq 1, V \in \Hht^1)$. 
Subtracting two identities 
\begin{align*}
x^k v \ap z_l w & = z_{k+p} V \ap z_l w \\
                         & = z_{k+p} (V \ap z_l w) + z_l (z_{k+p} V \ap w) + (z_{k+p} \cp z_l) (V \ap w), 
\end{align*}
and 
$$x^k(v \ap z_l w) = x^k z_p (V \ap z_l w) + x^k z_l (z_pV \ap w) + x^k (z_p \cp z_l)(V \ap w), $$
we have
\begin{align}\label{eq1}
x^k v \ap z_l w = x^k (v \ap z_l w) + z_l(x^k v \ap w) -z_{k+l}(v \ap w). 
\end{align}
Note that \eqref{eq1} is valid even if $k=0$. 
On the other hand,  we find by definition that
$$z_k v \ap z_l w = z_k(v \ap z_l w) + z_l (z_k v \ap w) + (z_k \cp z_l)(v \ap w). $$
Adding $t \times$\eqref{eq1} and $\h t \times$(\eqref{eq1} for $k \mapsto k-1$) to this identity, 
we conclude (i). 

By replacing $v$ with $S_\h^t(v)$ in (i) and using Lemma \ref{Sgam} (ii), we obtain (ii) for $v \in \Hht y$.  If $v=1$, we have (ii) easily by the rule of the harmonic product $\ap$. 

The identity (iii) is proved as well as the proof of (i). 
For $v=z_pV, w=z_rW \ (p, r \geq 1, V, W \in \Hht^1)$, 
subtracting two identities 
$$x^k v \ap x^l w = z_{k+p} (V \ap z_{l+r} W) + z_{l+r} (z_{k+p} V \ap W) + (z_{k+p} \cp z_{l+r}) (V \ap W), $$
and 
$$x^l(x^k v \ap w) = x^l z_{k+p} (V \ap z_r W) + x^l z_{r} (z_{k+p}V \ap W) + x^l (z_{k+p} \cp z_r)(V \ap W), $$
we have
\begin{align}\label{eq3}
x^k v \ap x^l w = z_{k+p} (V \ap z_{l+r} W) + x^l(x^k v \ap w) -z_{k+p+l}(V \ap w). 
\end{align}
Note that \eqref{eq3} is valid even if $l=0$. 
On the other hand,  we find by definition that
$$x^k v \ap z_l w = z_{k+p} (V \ap z_l w) + z_l (x^k v \ap w) + (z_{k+p} \cp z_l)(V \ap w). $$
Adding $t \times$\eqref{eq3} and $\h t \times$(\eqref{eq3} for $l \mapsto l-1$) to this identity, we conclude (iii). 
\end{proof}

\begin{itemize}
\item 
The product $\cast_\h$ on $\Hht y$ is defined by 
$$z_i u \cast_\h z_j v 
= z_{i+j}(u \ap v)$$
for $i, j \geq 1$, $u, v \in \Hht^1$. 
We define the product $\tcast_\h$ on $\Hht y$ by
\begin{align}
u \tcast_\h v =(S_\h^t)^{-1} (S_\h^t(u) \cast_\h S_\h^t(v))\ \ (u, v \in \mH_{\h, t} y). \label{tch}
\end{align}
\end{itemize}
\begin{defi}
We define the $\Q[\h, t]$-bilinear product $\tast_\h$ on $\Hht^1$ by the recursive rule  
\begin{align*}
1 \tast_\h w  & =  w \tast_\h 1 = w, \\
 z_i u \tast_\h z_j v & =  z_i ( u \tast_\h z_j v) + z_j ( z_i u \tast_\h v) + (1-2t) (z_i \cp z_j)(u \tast_\h v) \\
& \quad + (t^2-t) z_i \cp z_j \cp (u \tast_\h v)
\end{align*}
for $i, j \geq 1$ and words $u, v, w \in \Hht^1$. 
 
\end{defi}
This product is commutative and associative and can be viewed as a generalization of the products $\tast$ in \cite{Y} and $\ap$ as above or in \cite{T}. 

\begin{prop}\label{Stast}
For $v, w \in \Hht^1$, we have
$$v \tast_{\h} w = (S_\h^t)^{-1}(S_\h^t(v) \ap S_\h^t(w)). $$ 
\end{prop}
\begin{proof}
From Lemma \ref{Sap} (i), we have 
$$z_k v \ap \gamma_\h^t(z_l) w = z_k \bigl( v \ap \gamma_\h^t(z_l) w \bigr) +\gamma_\h^t(z_l)(z_k v \ap w) +(1-t) (z_k \cp z_l) (v \ap w)$$
for $k, l \geq 1$. 
Adding $t \times$(Lemma \ref{Sap} (iii)) and $\h t \times$(Lemma \ref{Sap} (iii) for $k \mapsto k-1$) to this identity, we can calculate  
\begin{align*} 
& \gamma_\h^t(z_k) v \ap \gamma_\h^t(z_l) w \\
& \quad = t x^k z_p \bigl( V \ap \gamma_\h^t(z_l)w \bigr) + \h t x^{k-1}z_p\bigl( V \ap \gamma_\h^t(z_l)w \bigr) +z_k \bigl( v \ap \gamma_\h^t(z_l)w \bigr) \\
& \qquad + \gamma_\h^t(z_l) \big( \gamma_\h^t(z_k)v \ap w \bigr) + t(1-t)(z_{k+p} \cp z_l)(V \ap w) \\
&\qquad + \h t(1-t) (z_{k+p-1} \cp z_l) (V \ap w) +(1-t)(z_k \cp z_l)(v \ap w)\\
& \quad = \gamma_\h^t(z_k)\bigl( v \ap \gamma_\h^t(z_l)w \bigr) 
    + \gamma_\h^t(z_l) \big( \gamma_\h^t(z_k)v \ap w \bigr) \\
& \qquad -t(x^k + \h x^{k-1}) \Bigl( v \ap \gamma_\h^t(z_l)w - z_p\bigl( V \ap \gamma_\h^t(z_l)w \bigr) - (1-t)(z_{p} \cp z_l) (V \ap w) \Bigr) \\
& \qquad + (1-t)(z_k \cp z_l)(v \ap w). 
\end{align*}
Then using Lemma \ref{Sap} (iii) for $k=0$, it turns out that  
\begin{align}\label{eq4} 
\gamma_\h^t(z_k) v \ap \gamma_\h^t(z_l) w
& = \gamma_\h^t(z_k) \bigl( v \ap \gamma_\h^t(z_l) w \bigr) +\gamma_\h^t(z_l) \bigl( \gamma_\h^t(z_k) v \ap w\bigr)  \nonumber\\
& \quad + \bigl( -t( x^k+\h x^{k-1}) \gamma_\h^t(z_l) + (1-t) (z_k \cp z_l) \bigr) (v \ap w) 
\end{align}
holds. 
Put $S_\h^t(v)$ and $S_\h^t(w)$ instead of $v$ and $w$ in \eqref{eq4} respectively 
and apply $(S_\h^t)^{-1}$ to both sides. Then thanks to the recursive rule of $\tast_\h$, 
property for $\cp$, and Lemma \ref{Sgam} (ii), induction on total depth works to establish Proposition \ref{Stast}.  
\end{proof}
\begin{prop}[harmonic product formula for $t$-$q$MZVs] We find that the map $Z_q^t$ is a homomorphism with respect to the harmonic product $\tast_\h$, i.e., 
$$Z_q^t(v \tast_\h w) = Z_q^t(v) Z_q^t(w)$$
for $u, v \in \Hht^0$. 
\end{prop}
\begin{proof}
Because of \eqref{Ztq}, Proposition \ref{Stast} and the fact that the map $Z_q^0$ is a homomorphism respect to the harmonic product $\ap$ (see \cite{T} for example), we have
\begin{align*}
Z_q^t(v \tast_\h w)  = Z_q^0 S_\h^t (v \tast_\h w)
                                 = Z_q^0 (S_\h^t(v) \ap S_\h^t(w))
                                 = Z_q^0 (S_\h^t(v))  Z_q^0(S_\h^t(w))
                                 = Z_q^t(v) Z_q^t(w). 
\end{align*}
\end{proof}
\subsection{Proof of Theorem \ref{tqKawa}}
When $t=0$, due to \cite[Theorem $4.6$]{T}, 
we have
\begin{align}\label{qKawa}
\sum_{\begin{subarray}{c} i+j=m \\ i, j\geq 1 \end{subarray}} Z_q^0(\varphi(v) \cast_\h y^i)Z_q^0(\varphi(w) \cast_\h y^j) =Z_q^0(\varphi(v \ast w) \cast_\h y^m) 
\end{align}
for any positive integer $m$ and any $v, w \in \Hht y$. 
Here, $\ast$ is the harmonic product for MZVs which is firstly introduced in \cite{H} (however its coefficient ring is extended to $\Q[\h, t]$).  
By \eqref{Ztq}, \eqref{vp} and \eqref{tch}, 
\begin{align*}
 {\rm LHS\ of\ \eqref{qKawa}} 
 & =  \sum_{\begin{subarray}{c} i+j=m \\ i, j\geq 1 \end{subarray}} 
Z_q^t (S_\h^t)^{-1} \bigl( S_\h^t \vp_\h^t (S_\h^t)^{-1}(v) \cast_\h y^i \bigr) 
Z_q^t (S_\h^t)^{-1} \bigl( S_\h^t \vp_\h^t (S_\h^t)^{-1}(w) \cast_\h y^j \bigr) \\
& = \sum_{\begin{subarray}{c} i+j=m \\ i, j\geq 1 \end{subarray}} 
Z_q^t \bigl( \vp_\h^t (S_\h^t)^{-1}(v) \tcast_\h (S_\h^t)^{-1}(y^i) \bigr)
Z_q^t \bigl( \vp_\h^t (S_\h^t)^{-1}(w) \tcast_\h (S_\h^t)^{-1}(y^j) \bigr)\\
& = \sum_{\begin{subarray}{c} i+j=m \\ i, j\geq 1 \end{subarray}} 
Z_q^t \bigl( \vp_\h^t (S_\h^t)^{-1} (v) \tcast_\h (-tx+y-\h t)^{i-1}y \bigr)\\
& \qquad \qquad \times Z_q^t \bigl( \vp_\h^t (S_\h^t)^{-1} (w) \tcast_\h (-tx+y-\h t)^{j-1}y \bigr).  
\end{align*}
Likewise, we have 
\begin{align*}
{\rm RHS\ of\ \eqref{qKawa}} 
& = -Z_q^t(S_\h^t)^{-1}(S_\h^t \vp_\h^t (S_\h^t)^{-1} (v \ast w) \cast_\h y^m)\\
& = -Z_q^t(\vp_\h^t (S_\h^t)^{-1} (v \ast w) \tcast_\h (S_\h^t)^{-1} (y^m)) \\
& = -Z_q^t(\vp_\h^t (S_\h^t)^{-1} (v \ast w) \tcast_\h (-tx+y-\h t)^{m-1} y)). 
\end{align*}
Hence we obtain Theorem \ref{tqKawa}. 

\begin{rmk}
By setting $t=1$ in Theorem \ref{tqKawa}, we have Kawashima type relation for $q$MZSVs. 
Taking the limit as $q \to 1$, we have Kawashima type relation for $t$-MZVs:
\begin{align*}
\sum_{\begin{subarray}{c} i+j=m \\ i, j\geq 1 \end{subarray}} Z^t(\varphi^t(v) \tcast (-tx+y)^{i-1}y)Z^t(\varphi^t(w) \tcast (-tx+y)^{j-1}y) &\\
=-Z^t(\varphi^t(v \tast w) \tcast (-tx+y)^{m-1}y)   
\end{align*}
for any positive integer $m$ and any $v, w \in \Ht y$. 
Here $Z^t, \vp^t, \tcast$, $\tast$ and $\Ht$ are regarded as each of $Z_q^t, \vp_\h^t, \tcast_\h$, $\tast_\h$ and $\Hht$ by assuming $\h = 0$ and $q \to 1$. 
This is established in \cite{TW}. 
\end{rmk}
\section{Cyclic sum formula for $t$-$q$MZVs}\label{tqC}
\subsection{Algebraic setup}
Let $n$ be a positive integer. We denote the $\Hht$-bimodule structure on ${\Hht}^{\otimes (n+1)}$ by ``$\diamond$'' defined by
\begin{align*}
a \diamond (w_1 \otimes w_2 \otimes \cdots \otimes w_n \otimes w_{n+1}) 
& = w_1 \otimes w_2 \otimes \cdots \otimes w_n \otimes a w_{n+1}, \\
(w_1 \otimes w_2 \otimes \cdots \otimes w_n \otimes w_{n+1})  \diamond b
& = w_1 b \otimes w_2 \otimes \cdots \otimes w_n \otimes w_{n+1} 
\end{align*}
for $a, b, w_1, w_2, \ldots, w_{n+1} \in \Hht$. 
For a positive integer $n$, we define the $\Q[\h, t]$-linear map 
$\CC_{n, t}^{(\h)}$ : $\Hht \longrightarrow {\Hht}^{\otimes (n+1)}$ by
$$\CC_{n, t}^{(\h)}(x) = -\CC_{n, t}^{(\h)}(y) = x \otimes ((1-t)x + y -\h t)^{\otimes (n-1)} \otimes y$$
and Leibniz rule
$$\CC_{n, t}^{(\h)}(v w) = \CC_{n, t}^{(\h)}(v) \diamond (\gamma_\h^t)^{-1}(w)
                                           + (\gamma_\h^t)^{-1}(v) \diamond \CC_{n, t}^{(\h)}(w)$$
for any $v, w \in \Hht$. Note that $\CC_{n, t}^{(\h)}(1)=0$.                                           
Let $M_n$ : $\Hht^{\otimes (n+1)} \longrightarrow \Hht$ denotes the multiplication map, i.e., 
$$M_n(w_1 \otimes w_2 \otimes \cdots \otimes w_n \otimes w_{n+1}) = w_1 w_2 \cdots w_n w_{n+1}. $$
We put $\rho_{n, t}^{(\h)} = M_n \CC_{n, t}^{(\h)}$. 

\begin{lem}\label{rhoS}
$\rho_{n, 0}^{(\h)} = S_\h^t \rho_{n, t}^{(\h)}. $
\end{lem}
\begin{proof}
Because of the linearity, it is enough to calculate $\CC_{n, t}^{\h} (w)$ for $w=z_{k_1} z_{k_2} \cdots z_{k_l} x^m$ ($l \geq 0$, $k_1, k_2, \ldots, k_l \geq 1$ and $m \geq 1$). 
By definition of $\CC_{n, t}^{(\h)}$, 

\begin{align*}
\CC_{n, t}^{(\h)}(w)
 & =  \sum_{i=1}^{l} \sum_{j=1}^{k_i-1} 
        (\gamma_\h^t)^{-1}(x^{k_1-1}y \cdots x^{k_{i-1}-1}y x^{j-1}) 
          \diamond \CC_{n, t}^{(\h)}(x) 
         \diamond 
        (\gamma_\h^t)^{-1}(x^{k_i-j-1}y x^{k_{i+1}-1}y \cdots x^{k_l-1}y x^{m})\\
    & \quad +\sum_{i=1}^{l} 
        (\gamma_\h^t)^{-1}(x^{k_1-1}y \cdots x^{k_{i-1}-1}yx^{k_i-1}) 
          \diamond \CC_{n, t}^{(\h)}(y) \diamond 
        (\gamma_\h^t)^{-1}(x^{k_{i+1}-1}y \cdots x^{k_l-1}y x^{m})\\
    & \quad +\sum_{j=1}^{m} (\gamma_\h^t)^{-1}(x^{k_1-1}y \cdots x^{k_l-1}y x^{j-1}) 
          \diamond \CC_{n, t}^{(\h)}(x) \diamond (\gamma_\h^t)^{-1}(x^{m-j}) \\
 & = \sum_{i=1}^{l} \sum_{j=1}^{k_i-1} 
        x \c (\gamma_\h^t)^{-1}(x^{k_i-j-1}y x^{k_{i+1}-1}y \cdots x^{k_l-1}y x^{m})\\
      & \qquad  \otimes ((1-t)x +y-\h t)^{\otimes (n-1)} \otimes 
          (\gamma_\h^t)^{-1}(x^{k_1-1}y \cdots x^{k_{i-1}-1}y x^{j-1}) y
        \\
    & \quad - \sum_{i=1}^{l} 
        x \c (\gamma_\h^t)^{-1}(x^{k_{i+1}-1}y \cdots x^{k_l-1}y x^{m})\\
        & \qquad \otimes ((1-t)x +y-\h t)^{\otimes (n-1)} \otimes 
        (\gamma_\h^t)^{-1}(x^{k_1-1}y \cdots x^{k_{i-1}-1}yx^{k_i-1}) y\\
    & \quad +\sum_{j=1}^{m} 
        x \c (\gamma_\h^t)^{-1}(x^{m-j})
         \otimes ((1-t)x +y-\h t)^{\otimes (n-1)} \otimes 
           (\gamma_\h^t)^{-1}(x^{k_1-1}y \cdots x^{k_l-1}y x^{j-1})y. 
\end{align*} 
Hence we obtain
\begin{align*}
S_\h^t \rho_{n, t}^{(\h)} (w) 
 & = S_\h^t M_n \CC_{n, t}^{\h} (w) \\
 & = \sum_{i=1}^{l} \sum_{j=1}^{k_i-1} 
       \gamma_\h^t(x) \c x^{k_i-j-1}y x^{k_{i+1}-1}y \cdots x^{k_l-1}y x^m \c \{ \gamma_\h^t((1-t)x+y-\h t)\}^{n-1}\\
       & \qquad \times x^{k_1-1}y \cdots x^{k_{i-1} -1}y x^{j-1}y \\
     & \quad - \sum_{i=1}^{l}\gamma_\h^t(x) \c x^{k_{i+1}-1}y \cdots x^{k_l-1} y x^m \c 
     \{ \gamma_\h^t((1-t)x+y-\h t)\}^{n-1} \c x^{k_1-1}y \cdots x^{k_i-1}y\\
     & \quad + \sum_{j=1}^{m} \gamma_\h^t(x) \c x^{m-j} \c \{ \gamma_\h^t((1-t)x+y-\h t)\}^{n-1} \c x^{k_1-1}y \cdots x^{k_l-1}y x^{j-1}y\\
 & = \sum_{i=1}^{l} \sum_{j=1}^{k_i-1} 
       x^{k_i-j}y x^{k_{i+1}-1}y \cdots x^{k_l-1}y x^m (x+y)^{n-1} x^{k_1-1}y \cdots x^{k_{i-1} -1}y x^{j-1}y \\
     & \quad - \sum_{i=1}^{l} 
       x \c x^{k_{i+1}-1}y \cdots x^{k_l-1} y x^m  (x+y)^{n-1} x^{k_1-1}y \cdots x^{k_i-1}y\\
     & \quad + \sum_{j=1}^{m} 
            x^{m-j+1} (x+y)^{n-1} x^{k_1-1}y \cdots x^{k_l-1}y x^{j-1}y\\
            = & M_n \CC_{n, 0}^{(\h)}(w) = \rho_{n, 0}^{(\h)}(w).    
\end{align*}
This completes the proof.  
\end{proof}
\subsection{Proof of Theorem \ref{tqCSF}}
First we prove the next two propositions. 
\begin{prop}\label{linKawa}
$\rho_{n, t}^{(\h)}(\check{\mH}_{\h, t}^1) \subset \ker Z_q^t$, where $\check{\mH}_{\h, t}^1$ denotes the subvector space of $\Hht^1$ generated by words of $\Hht^1$ expect for powers of $y$.  
\end{prop}
\begin{proof}
According to \cite[Proposition $2.5$]{TW0}, we have 
\begin{align}\label{rho}
\rho_{n, 0}^{(\h)}(\check{\mH}_{\h, t}^1) \subset L_x \vp (\Hht y \ast \Hht y), 
\end{align}
where $L_x$ is the left multiplication by $x$ defined by $L_x(w)=xw$ for any $w \in \Hht$. 
By \eqref{rho} and Lemma \ref{rhoS}, we have 
\begin{align*}
\rho_{n, t}^{(\h)}(\check{\mH}_{\h, t}^1) 
& =  (S_\h^t)^{-1} \rho_{n, 0}^{(\h)} (\check{\mH}_{\h, t}^1) \subset (S_\h^t)^{-1} (L_x \vp(\Hht y \ast \Hht y)). 
\end{align*} 
Also we find that 
\begin{align}\label{Lx}
S_\h^t L_x = L_x S_\h^t
\end{align}
by definition of $L_x$ and $S_\h^t$. 
By \eqref{vp} and \eqref{Lx}, we obtain 
\begin{align*}
\rho_{n, t}^{(\h)}(\check{\mH}_{\h, t}^1) 
& \subset L_x \vp_{\h}^t (S_\h^t)^{-1}(\Hht y \ast \Hht y) = \vp_{\h}^t (S_\h^t)^{-1}(\Hht y \ast \Hht y) \tcast_\h y. 
\end{align*} 
Therefore we conclude the proposition because of Theorem \ref{tqKawa} for the case of  $m=1$. 
\end{proof}
\begin{prop}\label{rho2}
For cyclically equivalent words $v, w \in \Hht$, we have $\rho_{1, t}^{(\h)}(v) = \rho_{1, t}^{(\h)}(w).$
\end{prop}
\begin{proof}
Let $u_1, u_2, \ldots, u_l \in \{ x, y \}$ and ${\rm sgn} (u) = 1$ or $-1$ according to $u=x$ or $y$. 
Because of  
$$\CC_{1, t}^{(\h)} (u) = {\rm sgn}(u)(x \otimes y)$$
for $u \in \{ x, y \}$, we have
\begin{align*}
\CC_{1, t}^{(\h)}(u_1 u_2 \cdots u_l) 
& = \sum_{i=1}^{l} (\gamma_\h^t)^{-1} (u_1 \cdots u_{i-1}) \diamond \CC_{1, t}^{(\h)} (u_i) \diamond (\gamma_\h^t)^{-1} (u_{i+1} \cdots u_l)\\
& = \sum_{i=1}^{l} {\rm sgn}(u_i) x \c (\gamma_\h^t)^{-1} (u_{i+1} \cdots u_l) \otimes (\gamma_\h^t)^{-1} (u_1 \cdots u_{i-1}) \c y, 
\end{align*}
where we assume $u_1 \cdots u_{i-1}=1$ if $i=1$ and $u_{i+1} \cdots u_l =1$ if $i=l$. 
Therefore we obtain 
\begin{align*}
\rho_{1, t}^{(\h)}(u_1 u_2 \cdots u_l) 
& = \sum_{i=1}^{l} {\rm sgn}(u_i) x \c (\gamma_\h^t)^{-1} (u_{i+1} \cdots u_l u_1 \cdots u_{i-1}) \c y.  
\end{align*} 
Since the right-hand side does not change under the cyclic permutations of $\{ u_1, u_2, \ldots, u_l \}$, we conclude the proposition.  
\end{proof}
Now we prove Theorem \ref{tqCSF}. 
We calculate 
\begin{align*}
\CC_{1, t}^{(\h)} & (\gamma_\h^t (z_{k_1} z_{k_2} \cdots z_{k_l}))\\
 & =  \CC_{1, t}^{(\h)} (x^{k_1-1}(tx+y+\h t) \cdots x^{k_l-1}(tx+y+\h t))\\
 & =  \sum_{i=1}^{l} \sum_{j=0}^{k_i-2} 
      (\gamma_\h^t)^{-1}( x^{k_1-1}(tx+y+\h t) \cdots x^{k_{i-1}-1}(tx+y+\h t) x^j) 
       \diamond \CC_{1, t}^{(\h)}(x) \\
     &  \qquad \diamond (\gamma_\h^t)^{-1}( x^{k_i-j-2}(tx+y+\h t)x^{k_{i+1}-1}(tx+y+\h t) \cdots x^{k_l-1}(tx+y+\h t) )\\
     & \quad 
     + \sum_{i=1}^{l} (\gamma_\h^t)^{-1} ( x^{k_1-1}(tx+y+\h t) \cdots x^{k_{i-1}-1}(tx+y+\h t) x^{k_{i}-1}) 
     \diamond \CC_{1, t}^{(\h)}(tx+y+\h t) \\
     & \qquad \diamond (\gamma_\h^t)^{-1}( x^{k_{i+1}-1}(tx+y+\h t) \cdots x^{k_l-1}(tx+y+\h t) )\\
& = \sum_{i=1}^{l} \sum_{j=0}^{k_i-2} x^{k_i-j-1}y x^{k_{i+1}-1} y \cdots x^{{k_l}-1} y \otimes x^{k_1-1}y \cdots x^{k_{i-1}-1}y x^{j} y\\
& \quad  
     + (t-1) \sum_{i=1}^{l} x \c x^{k_{i+1}-1} y \cdots x^{k_{l}-1} y \otimes x^{k_{1}-1} y\cdots x^{k_i-1} y  
\end{align*}
and 
$$\CC_{1, t}^{(\h)} \bigl( x^{k-l}(x + \h)^l \bigr)
= \sum_{i=0}^{l} \binom{l}{i} \h^{i} \sum_{j=1}^{k-i} x^{k-i-j+1} \otimes x^{j-1} y, $$
where $k=k_1 + k_2+ \cdots + k_l$. 
Hence we have
\begin{align*}
\rho_{1, t}^{(\h)} (\gamma_\h^t (z_{k_1} z_{k_2} \cdots z_{k_l}) - t^l 
 x^{k-l}(x + \h)^l )
  = &  \sum_{i=1}^{l} \sum_{j=0}^{k_i-2} z_{k_i-j} z_{k_{i+1}} \cdots z_{{k_l}} z_{k_1} \cdots z_{k_{i-1}} z_{j+1}\\
& - (1-t) \sum_{i=1}^{l} z_{k_{i}+1} z_{k_{i+1}} \cdots z_{k_{l}} z_{k_{1}} \cdots z_{k_{i-1}}\\
& -  t^l \sum_{i=0}^{l} (k-i) \binom{l}{i} \h^i z_{k-i+1}. 
\end{align*}
If $(k_1, k_2, \ldots, k_l) \neq (1, 1, \ldots, 1)$, 
each term of $\gamma_\h^t (z_{k_1} z_{k_2} \cdots z_{k_l}) - t^l 
x^{k-l}(x + \h)^l$ modulo cyclic permutation can be regarded as an element in $\check{\mH}_{\h, t}^1$. 
Therefore we obtain Theorem \ref{tqCSF} by Proposition \ref{linKawa} and \ref{rho2}. 
\section{Hoffman type relation for $t$-$q$MZVs}\label{tqH}
Finally, as another application of Theorem \ref{tqKawa} for $m=1$, 
we show that Theorem \ref{tqKawa} includes the following Hoffman type relation for $t$-$q$MZVs. 
\begin{thm}
For positive integers $k_1, k_2, \ldots, k_l$ with $k_1 \geq 2$, we have
\begin{align*}
& \sum_{i=1}^{l}\sum_{j=0}^{k_i-2} \tq(k_1, \ldots, k_{i-1}, k_i-j, j+1, k_{i+1}, \ldots, k_l)\\
& = \sum_{i=1}^{l} \{1+ (k_i-2+ \delta_{i, l})t\} \tq(k_1, \ldots, k_{i-1}, k_i+1, k_{i+1}, \ldots, k_l)\\
& \ \ \ \ + t(t-1) \sum_{i=1}^{l-1} \tq(k_1, \ldots, k_{i-1}, k_i+k_{i+1}+1, k_{i+2}, \ldots, k_l)\\
& \ \ \ \ +(1-q) \sum_{i=1}^{l}\left\{ t (k_i-1) \tq (k_1, \ldots, k_l) 
+t(t-1)\tq(k_1, \ldots, k_{i-1}, k_i+k_{i+1}, k_{i+2}, \ldots, k_l) \right\},  
\end{align*}
where $\delta_{i, l}$ stands for Kronecker's delta.
\end{thm}
\begin{proof}
Let $\partial_1$ : $\Hht \longrightarrow \Hht$ denotes the derivation
determined by 
$$\partial_1(x) = -\partial_1(y) = xy. $$
Calculate 
\begin{align*}
(S_\h^t)^{-1} \partial_1 S_\h^t (z_{k_1} \cdots z_{k_l}) 
\ \ (k_1 \geq 2), 
\end{align*}
which is known to be an element in $\ker Z_q^t$ 
because of \eqref{Ztq}, $S_\h^t (z_{k_1} \cdots z_{k_l}) \in \Hht^0$ and $\partial_1(\Hht^0) \subset \ker Z_q$ (see \cite{B} for the last one), and we obtain the theorem.  
\end{proof}
We notice that when $t=0$ this theorem is reduced to Hoffman type relation for $q$MZVs  proved in \cite{B}. 
Taking the limit as $q \to 1$, we have Hoffman type relation for $t$-MZVs proved in \cite{W} 
by another method using the double shuffle product structure for $t$-MZVs. 

\subsubsection*{Acknowledgments}
The author is supported by the Grant-in-Aid for Young Scientists (B) No. 15K17523, Japan Society for the Promotion of Science. 


\begin{thebibliography}{9}
\bibitem{B} D. M. Bradley, Multiple $q$-zeta values, J. Algebra 283 (2005), 752--798.

\bibitem{H} M. Hoffman, The algebra of multiple harmonic series, J. Algebra 194 (1997), 477--495.

\bibitem{OO} Y. Ohno, J. Okuda, On the sum formula for the $q$-analogue of non-strict multiple zeta values, Proc. Amer. Math. Soc. 135 (2007), 3029--3037.

\bibitem{OT} J. Okuda, Y. Takeyama, On relations for the multiple $q$-zeta values, Ramanujan J. 14 (2007), 379--387. 

\bibitem{T} Y. Takeyama, Quadratic relations for a $q$-analogue of multiple zeta values, Ramanujan J. 14 (2012), 15--28. 

\bibitem{TW0} T. Tanaka, N. Wakabayashi, An algebraic proof of the cyclic sum formula for multiple zeta values, J. Algebra 323 (2010), 766--778.

\bibitem{TW} T. Tanaka, N. Wakabayashi, Kawashima's relations for interpolated multiple zeta values, J. Algebra 447 (2016), 424--431. 

\bibitem{Y}  S. Yamamoto, Interpolation of multiple zeta and zeta-star values, J. Algebra 385 (2013), 102--114.

\bibitem{W} N. Wakabayashi, Double shuffle and Hoffman's relations for interpolated multiple zeta values, preprint. 

\bibitem{Z} J. Zhao, Multiple $q$-zeta functions and multiple $q$-polylogarithms, Ramanujan J. 14 (2007), 189--221.


\end{thebibliography}
\end{document}